\documentclass[12pt]{article}
\usepackage{theorem,amsmath,amssymb,latexsym,color}
\usepackage[all]{xy}

\theoremstyle{plain}
\theoremheaderfont{\normalfont\bfseries}
\theorembodyfont{\itshape}

\newtheorem{theorem}{Theorem}[section]
\newtheorem{proposition}[theorem]{Proposition}

\newtheorem{conjecture}[theorem]{Conjecture}

\theorembodyfont{\upshape}
{\theoremstyle{break}

}

\theorembodyfont{\upshape}
\newtheorem{remark}[theorem]{Remark}

\theorembodyfont{\upshape}

\theorembodyfont{\upshape} 
\newtheorem{definition}[theorem]{Definition}

\theorembodyfont{\upshape}
\newtheorem{question}[theorem]{Question}

\newtheorem{exam}[theorem]{Example}

\theorembodyfont{\itshape}

\theorembodyfont{\upshape}

\theorembodyfont{\upshape}

\theorembodyfont{\upshape}

\theorembodyfont{\upshape}

\theorembodyfont{\upshape}

\theorembodyfont{\upshape}
\newtheorem{proof}{Proof}

\theorembodyfont{\upshape}

\newcommand{\myemail}[1]{\indent \emph{E-mail:} {\tt #1}}
\newcommand{\myaddress}[1]{\indent {\sc #1}\par}
\newcommand{\fbar}{\overline{f}}

\newcommand{\bY}{{\mathbf Y}}
\newcommand{\bX}{{\mathbf X}}
\newcommand{\bZ}{{\mathbf Z}}
\newcommand{\bE}{{\mathbf E}}
\newcommand{\ix}{\mathcal {X}}

\newcommand{\iu}{\mathcal {U}}
\newcommand{\iz}{\mathcal {Z}}
\newcommand{\iy}{\mathcal {Y}}
\newcommand{\ie}{\mathcal {E}}
\newcommand{\ic}{\mathcal {C}}
\newcommand{\is}{\mathcal {S}}
\newcommand{\Q}{\mathbb{Q}}

\newcommand{\A}{\mathbb{A}}
\newcommand{\G}{\mathbb{G}}
\newcommand{\Z}{\mathbb{Z}}
\newcommand{\Pro}{\mathbb{P}}
\newcommand{\N}{\mathcal{N}}
\DeclareMathOperator{\Spec}{Spec}
\DeclareMathOperator{\Proj}{Proj}

\title{Strong cycles and intersection products on good moduli spaces}
\author{Dan Edidin\footnote{Research supported in part by
Simons Collaboration Fellowship 315460.}\; and Matthew Satriano\footnote{Research supported in part by NSERC grant RGPIN-2015-05631.}}
\date{\today}

\begin{document}
\maketitle
\begin{abstract}
  We introduce conjectures relating the Chow ring of a smooth Artin
  stack $\ix$ to the Chow groups of its possibly singular good moduli
  space $\bX$. In particular, we conjecture the existence of an intersection
  product on a subgroup of the full Chow group $A^*(\bX)$
  coming from {\em strong cycles} on $\ix$.
\end{abstract}
\section{Introduction}
Let $X$ be a non-singular projective variety with an action of a linearly reductive group $G$ over an algebraically closed field of characteristic 0. 
Given a linearization $L$ of the action we can define the open set,
$X^s$, of $L$-stable points and the open set, $X^{ss}$, of $L$-semistable
points.  Mumford's GIT produces quotients
$X^s/G$ and $X^{ss}/G$ of these open sets. 
The former quotient has mild (finite quotient)
singularities but is not in general proper. The latter quotient is
projective and  contains
$X^s/G$ as an open set, but in general has worse singularities.

When the action of $G$ has generically finite stabilizers the GIT
quotient $X^{s}/G$ is the {\em coarse moduli space} of the
Deligne-Mumford (DM) quotient stack $[X^s/G]$. Likewise, the quotient
$X^{ss}$ is the {\em good moduli space} in the sense \cite{Alp:13} of
the Artin quotient stack $[X^{ss}/G]$.

There is a fully developed intersection theory on quotient stacks \cite{EdGr:98} which assigns to any smooth quotient stack $\ix$ a Chow ring $A^*(\ix)$. When $\ix$ is smooth Deligne-Mumford, there is a beautiful relationship between the Chow ring of $\ix$ and that of its coarse moduli space $\bX$: there is a pushforward isomorphism on rational Chow groups $A^*(\ix)_\Q \stackrel{\simeq}{\longrightarrow} A^*(\bX)_\Q$. As a result, the possibly singular variety (or more generally algebraic space) $\bX$ has an intersection product on its rational Chow groups induced from the intersection product on the Chow groups of $\ix$. Furthermore, a fundamental result of Vistoli \cite{Vis:89} states that any variety $\bX$ with finite quotient singularities is the coarse moduli space of a smooth Deligne-Mumford stack $\ix$, and hence by the above, $A^*(\bX)_\Q$ carries an intersection product coming from that of $A^*(\ix)_\Q$.

However, for varieties $\bX$ with worse than finite quotient
singularities, or for stacks $\ix$ which are not Deligne-Mumford, the
beautiful picture above breaks down in almost every aspect. In
general, if $\ix$ is a smooth Artin stack, the rational Chow groups
$A^*(\ix)_\Q$ can be non-zero in arbitrarily high degree, so cannot be
isomorphic to the rational Chow groups $A^*(\bX)_\Q$ of the good
moduli space. In fact, it is not even known if the moduli map
$\pi\colon\ix\to\bX$ induces a pushforward map $\pi_*\colon
A^*(\ix)_\Q\to A^*(\bX)_\Q$. Moreover, there are good quotients by actions
of reductive groups, such as the cone over a quadric hypersurface,
where one can prove there is no reasonable intersection product on the
Chow groups.

In this article we consider two questions about the Chow groups of Artin stacks and their good moduli spaces aimed at rectifying the above problems.

\begin{question} \label{ques.pushforward}
Let $\ix$ be an Artin stack with good moduli space morphism $\pi \colon \ix \to \bX$. Is there a geometrically meaningful pushforward map $\pi_* \colon A_*(\ix)_\Q \to A_*(\bX)_\Q$?
\end{question}

Since this pushforward map, if it exists, cannot be an isomorphism we are led to the following:

\begin{question} \label{ques.subring} Is there an interesting subring
  of the Chow ring $A^*(\ix)$ such that the restriction of $\pi_*$ to
  this subring is injective?
\end{question}

\begin{remark}[{Application to Chow groups of singular varieties}]
  As mentioned above, there exist singular varieties $\bX$ that are
  good quotients by reductive groups for which $A^*(\bX)_\Q$ carries
  no reasonable intersection product, e.g.~the cone over a quadric
  hypersurface. However, if Questions \ref{ques.pushforward} and
  \ref{ques.subring} have affirmative answers, then we can identify an
  interesting \emph{subgroup} of $A^*(\bX)_\Q$ that \emph{does} carry
  an intersection product. Indeed, such a variety $\bX$ is a good
  moduli space of a smooth Artin stack $\ix$; by Question
  \ref{ques.pushforward}, we have a map $\pi_* \colon A^*(\ix)_\Q \to
  A^*(\bX)_\Q$ and our desired subgroup of $A^*(\bX)_\Q$ is the image
  of the subring provided by Question \ref{ques.subring}.
\end{remark}

We answer Question \ref{ques.pushforward} for the class of good moduli
space morphisms that look \'etale locally like GIT quotients with
non-empty stable loci; this is done in Section
\ref{sec:Reichstein}. We then give a conjectural answer to Question
\ref{ques.subring} in Section \ref{sec:strong-Chow}. We shall see that
the answers to both questions are related by the concept of {\em
  strong embedding} which was the subject of the first author's TIFR
Colloquium lecture.

\vspace{1em}

\noindent{\bf Acknowledgments} It is a pleasure to thank the organizers of the TIFR International Colloquium on $K$-theory for a wonderful conference and a stimulating environment.

\subsection{Background on stacks and good moduli spaces}
For simplicity of exposition, all stacks are assumed to be defined
over an algebraically closed field of characteristic 0. We also assume
that any stack is of finite type over the ground field and has affine
diagonal.  The following definitions are generalizations of concepts
in invariant theory.

\begin{definition}\cite[Definition 4.1]{Alp:13}
Let $\ix$ be an algebraic stack and let $\bX$ be an algebraic space. We say
that $\bX$ is a {\em good moduli space of $\ix$} if there is a morphism
$\pi \colon \ix \to \bX$ such that
\begin{enumerate}
\item $\pi$ is {\em cohomologically affine} meaning that the pushforward functor $\pi_*$
on the category of quasi-coherent ${\mathcal O}_\ix$-modules is exact.

\item $\pi$ is {\em Stein} meaning that the natural map ${\mathcal O}_\bX \to \pi_* {\mathcal O}_\ix$ is an isomorphism.
\end{enumerate}
\end{definition}
\begin{remark} By \cite[Theorem 6.6]{Alp:13}, a good moduli space morphism
$\pi\colon\ix\to\bX$ is the universal morphism from $\ix$ to an algebraic space. That is, if $\bZ$ is an algebraic space then any morphism $\ix \to \bZ$ factors through a morphism $\bX \to \bZ$. Consequently $\bX$ is unique up to unique isomorphism, so we will refer to $\bX$ as \emph{the} good moduli space of $\ix$.
\end{remark}

\begin{remark}
If $\ix = [X/G]$ where $G$ is a linearly reductive algebraic group
then the statement that $\bX$ is a good moduli space for $\ix$ is 
equivalent to the
statement that $\bX$ is the good quotient of $X$ by $G$. 
%Indeed a theorem of
%Alper, Hall, and Rydh \cite{AHR:15} states that if $\pi \colon\ix \to \bX$ is a
%good moduli space morphism and if $x$ is a closed point the stabilizer, $G_x$,
%of $x$ is linearly reductive and the map $\pi$ looks \'etale locally like
%the $[\Spec A/G_x] \to \Spec A^{G_x}$.
\end{remark}

\begin{definition} \label{def.stablegms}  \cite{EdRy:16}
Let $\ix$ be an Artin with good moduli space $\bX$ and
let $\pi \colon \ix \to \bX$ be the good moduli space morphism. We say that a closed point
of  $\ix$ is
{\em
   stable} if $\pi^{-1}(\pi(x)) = x$ under the induced map of
  topological space $|\ix| \to |\bX|$. A point $x$ of $\ix$ is {\em
    properly stable} if it is stable and the stabilizer of $x$ is finite.

We say $\ix$ is  stable (resp.~properly stable) if there is a good moduli
space $\ix \stackrel{\pi} \to \bX$ and the 
the set of stable (resp.~properly stable) points is non-empty.
\end{definition}
\begin{remark}
This definition is modeled on GIT. If $G$ is a linearly reductive group
and $X^{ss}$ is the set of semistable points for a linearization of the 
action of $G$ on a projective variety $X$ then a (properly) stable point
of $[X^{ss}/G]$ corresponds to a (properly) stable orbit in the sense of GIT. 
The stack $[X^{ss}/G]$ is stable if and only if $X^{s} \neq \emptyset$. Likewise
$[X^{ss}/G]$ is properly stable if and only if $X^{ps} \neq \emptyset$. As is the case for GIT quotients, the set of stable (resp. properly stable points)
is open \cite{EdRy:16}.

We denote by $\ix^s$ (resp.~$\ix^{ps}$) the open substack of $\ix$ of
stable (resp.~properly stable) points. The stack $\ix^{ps}$ is the
maximal Deligne-Mumford substack of $\ix$ which is saturated with
respect to the good moduli space morphism $\pi\colon\ix \to X$.  In
particular, a stack $\ix$ with good moduli space $\bX$ is properly
stable if and only if it contains a non-empty {\em saturated}
Deligne-Mumford open substack. 
\end{remark}
\begin{exam}
  Consider the action of $\G_m$ on $\A^2 = \Spec k[x,y]$ with
  weights $(1,0)$. The quotient stack $\ix = [\A^2/\G_m]$ is one
  dimensional and has one-dimensional good moduli space $\A^1 = \Spec
  k[y]$. The open set $\A^2 \smallsetminus V(x)$, is the maximal open
  set on which $\G_m$ acts with finite (in fact trivial)
  stabilizers. Hence $\ix$ has a maximal open DM substack (which in this
  case is a scheme) $\iu = [(\A^2 \smallsetminus V(x)) /\G_m] =
  \A^1$. However, this open substack is not saturated with respect to
  the good moduli space morphism $\ix \to \A^1$. Indeed, this action
  of $\G_m$ on $\A^2$ has no stable or properly stable points and
  $\ix$ is not a stable stack.
\end{exam}

\subsection{Background on equivariant
Chow groups and Chow groups of stacks}

\subsubsection{Equivariant Chow groups}
If $X$ is an equidimensional scheme or algebraic space, we use the
notation $A^k(X)$ to denote the Chow group of codimension-$k$ cycles
modulo rational equivalence. The total Chow group $A^*(X)$ is the
direct sum $\oplus_{k=0}^{\dim X} A^k(X)$. When $X$ is smooth the
intersection product makes $A^*(X)$ into a graded ring.

The definition of equivariant Chow groups is
modeled on the Borel construction in equivariant cohomology. If a linear
algebraic group $G$ acts on $X$ then the $k$-th equivariant Chow group
$A^k(X)$ is defined to be $A^k(X_G)$ where  where $X_G$ is any quotient of the form $(X \times
U)/G$ where $U$ is an open set in a representation ${\mathbf V}$ of
$G$ such that $G$ acts freely on $U$ and ${\mathbf V} \smallsetminus
U$ has codimension more than $i$.  In \cite{EdGr:98} it is shown that
such pairs $(U, {\mathbf V})$ exist for any algebraic group and that
the definition of $A^k_G(X)$ is independent of the choice of $U$
and ${\mathbf V}$.

Note that, since representations can have arbitrarily high dimension,
$A^k_G(X)$ can be non-zero in arbitrarily high degree. Thus the total
equivariant Chow group $A^*_G(X)$ is the {\em infinite} direct sum
$\oplus_{k=0}^\infty A^k_G(X)$. An equivariant $k$-cycle need not be supported on $X$, but only on $X \times {\mathbf V}$ where ${\mathbf V}$ is a representation of $G$.

Because equivariant Chow groups are defined as Chow groups of schemes (or more generally algebraic spaces)
they enjoy all of the functoriality of ordinary Chow groups. 
In particular, if $X$ is smooth then pullback along the diagonal
defines an intersection product on the total Chow group
$A^*_G(X)$.

\subsubsection{Chow groups of stacks}

Gillet \cite{Gil:84} and Vistoli \cite{Vis:89} defined Chow groups of Deligne-Mumford stacks in an analogous way to Chow groups of schemes. Namely, they considered the group generated by the classes of integral closed substacks modulo rational equivalences.  With rational coefficients this theory has many desired properties such as an intersection product on the rational Chow groups of smooth stacks. Kresch \cite{Kre:99} generalized their work by defining integral Chow groups for Artin stacks $\ix$ with quasi-affine diagonal. When $\ix=[X/G]$ is a quotient stack, Kresch's Chow groups $A^*(\ix)$ agree with the equivariant Chow groups $A^*_G(X)$.

%Gillet \cite{Gil:84} and Vistoli \cite{Vis:89} defined Chow groups of Deligne-Mumford stacks in an analogous way to Chow groups of schemes. Namely, they considered the group generated by the classes of integral closed substacks modulo rational equivalences.  With rational coefficients this theory has many desired properties such as an intersection product on the rational Chow groups of smooth stacks. For any quotient stack $\ix = [X/G]$ the equivariant Chow groups are independent of the presentation and therefore an invariant of the stack. This gives a well-behaved integral intersection theory for quotient stacks. In the Deligne-Mumford case, these groups are isomorphic to the groups of Gillet and Vistoli after tensoring with $\Q$. For Artin stacks with quasi-affine diagonal Kresch \cite{Kre:99} defined integral Chow groups which agree with the equivariant Chow groups when $\ix = [X/G]$ is a quotient stack.

When $\ix$ is Deligne-Mumford with coarse space $\bX$ the proper
pushforward $\pi_* \colon A^*(\ix)_\Q \to A^*(\bX)_\Q$ is an
isomorphism \cite{Vis:89,EdGr:98}. The pushforward is defined on cycles
by the formula
$[\iz] \mapsto e_{\iz}^{-1}[\pi(\iz)]$ where $e_{\iz}$ is the generic 
order of the stabilizer group along $\iz$. In particular this means that if
$\ix = [X/G]$ is a Deligne-Mumford stack then every equivariant Chow
class can be represented by a $G$-invariant cycle on $X$ (as opposed
to $X \times {\mathbf V}$ where ${\mathbf V}$ is a representation of
$G$).  Consequently $A^k(\ix)_\Q = 0$ for $k > \dim \ix$.

If $\ix$ is not Deligne-Mumford then $A^k(\ix)_\Q$ will be non-zero in
arbitrarily high degree, so if $\pi \colon \ix \to \bX$ is a good moduli
space morphism we cannot expect $A^*(\ix)_\Q$ to equal
$A^*(\bX)_\Q$.

\subsection{Strong lci morphisms of stacks with good moduli spaces}

We now introduce the key concepts of strong embeddings and strong lci morphisms.

\begin{definition}
  Let $\ix$ be an Artin stack with good moduli space 
$\pi \colon \ix \to \bX$. 
A closed embedding $\iy \to  \ix$ is {\em strong} if $\iy$ 
  stack-theoretically saturated with respect to the morphism $\pi$.
A regular embedding which is strong will be
  called a {\em strong regular embedding}.
\end{definition}

\begin{remark}
In \cite{Edi:16} the first author considered the notion of strong
regular embeddings of tame\footnote{A stack is tame if the stabilizer of every closed point is finite and  linearly reductive. In characteristic 0 a stack is tame if and only if it is Deligne-Mumford.} stacks.
\end{remark}

\begin{remark} Theorem 2.9 of \cite{AHR:15} states that the good
  moduli space morphism $\pi \colon \ix \to \bX$ looks \'etale locally
  like the morphism $[\Spec A/G] \to \Spec A^G$ where $G$ is a
  linearly reductive group acting on a finitely generated $k$-algebra
  $A$.  The condition that a closed embedding $\iy \to \ix$ is strong
  is equivalent to the assertion that the local ideals $I$ of $\iy$
  satisfy $I^G A = I$.
\end{remark}

Strong regular embeddings $\iy \to \ix$ are characterized by a number of
equivalent properties including
\begin{enumerate}
\item The morphism of good moduli spaces induced by the closed embedding $\iy \to \ix$ is a regular embedding and the diagram
$$\xymatrix{\iy \ar[d]\ar[r] & \ix\ar[d]\\ \bY \ar[r] & \bX}$$
is cartesian.

\item
  The stabilizer $G_y$ of any point $y$ of $\iy$ acts trivially on the fiber of the normal bundle $\N_{\iy/\ix,y}$.
\end{enumerate}

These facts follows from the proof of \cite[Theorem 2.2]{Edi:16} for tame stacks since the proof only uses the fact a tame stack is \'etale locally the quotient of an affine scheme by a linearly reductive group.

We next extend the notion of strong regular embedding to that of a strong lci morphism.
\begin{definition}
\label{def:strong-lci}
  A morphism of $\iy \to \ix$ is a {\em strong lci morphism} if it factors
  as $\iy \to \Pro(\ie) \to \ix$ where $\iy \to \Pro(\ie)$ is a strong
  regular embedding and $\ie$ is a vector bundle on $\ix$ such that
  the stabilizer of any point $x$ of $\ix$ acts trivially on the fiber $\ie_x$.
\end{definition}

\begin{proposition}
  If $ f \colon \iy \to \ix$ is a strong lci morphism and $\bX$ is a good moduli
  space for $\ix$ then $\iy$ has a good moduli space $\bY$. Moreover,
  the induced morphism of good moduli spaces $\bY \to \bX$ is lci and
  the diagram $$\xymatrix{\iy \ar[d]\ar[r] & \ix\ar[d]\\ \bY \ar[r] &
    \bX}$$ is cartesian.
\end{proposition}
\begin{proof}
  By assumption, the map $\iy\to\ix$ has a factorization
  $\iy\to\Pro(\ie)\to\ix$ as in Definition \ref{def:strong-lci}.  If
  $G_x$ acts trivially on the fiber $\ie_x$ for all points $x$ of
  $\ix$, then $\ie = \pi^*{\bE}$ where $\bE$ is a vector bundle on
  $\bX$.  It follows from \cite[Proposition 4.7]{Alp:13} that
  $\Pro(\bE)$ is the good moduli space of $\Pro(\ie)$. Since $\iy$ is
  strongly embedded in $\Pro(\bE)$ the proposition follows from the
  corresponding statement for strong regular embeddings of tame stacks
  \cite[Theorem 2.2]{Edi:16}.
\end{proof}

\section{Pushforwards in Chow groups: an answer to Question \ref{ques.pushforward}}
\label{sec:Reichstein}

The goal of this section is to prove the following theorem.
\begin{theorem} \label{thm.pushforward}
Let $\ix$ be a properly stable smooth Artin stack with good moduli space morphism $\pi\colon\ix\to\bX$.
Then there is a pushforward map $\pi_* \colon A^*(\ix)_\Q \to A^*(\bX)_\Q$
such that
\begin{enumerate}
\item If $x$ is a closed point with finite stabilizer group $G_x$, then $\pi_*[x] = |G_x|^{-1}[\pi(x)]$.

\item\label{item:stable-locus}  More generally, if $\iz$ is an irreducible closed substack of $\ix$
which is contained in the maximal saturated DM substack $\ix^{ps}$ then, $\pi_*[\iz] = e_\iz^{-1} [\pi(\iz)]$
where $e_\iz$ is the generic 
order of the stabilizer group along $\iz$. 

\item\label{item:functoriality}
The pushforward commutes with strong proper lci morphisms. Precisely,
if $f \colon \iy \to \ix$ is a strong lci morphism of smooth properly stable
Artin stacks\footnote{If $\ix$ is properly stable then $\iy$ is automatically
properly stable as well.} and $\fbar \colon \bY \to \bX$ is the induced map on good moduli spaces, then the diagram of Chow groups of stacks and good moduli spaces commutes
$$\xymatrix{A^*(\iy)_\Q \ar[d]^{\pi_{\iy,*}} \ar[r]^{f_*} & A^*(\ix)_\Q \ar[d]^{\pi_{\ix,*}} \\
A^*(\bY)_\Q \ar[r]^{\fbar_*} & A^*(\bX)_\Q}
$$
\end{enumerate}

\end{theorem}
\subsection{The example of \cite{EGS:13}} \label{sec.egs} An obvious question is whether the functoriality property (\ref{item:functoriality}) of Theorem \ref{thm.pushforward} holds for arbitrary lci morphisms of Artin stacks, as opposed to strong lci morphisms.  In \cite[Theorem 1]{EGS:13}, we showed that any choice of pushforward map that commutes with all regular embeddings must, in fact, be the 0 map. Thus, we cannot expect the functoriality property Theorem \ref{thm.pushforward} (\ref{item:functoriality}) to extend to arbitrary lci morphisms, let alone arbitrary regular embeddings. This shows the importance of the condition that the lci morphisms be strong.

Let us briefly recall the example of \cite{EGS:13}. Consider the action of $\G_m^3$ on $\A^5$ with weight matrix
$$\left(\begin{array}{ccccc}
1 &  0 & 0 & 1 & 1\\
0 & 1 & 0 & 1 & 1 \\
0 & 0 & 1 & 0 & 1
\end{array}\right)$$ so that $(s,t,u)\in\G_m^3$ acts by $(s,t,u) \cdot (x,y,z,w,v) = (sx,ty, uv, stz, stu w)$. The quotient of the open set $X = \A^5 \smallsetminus V(xyz, zw, v)$ is the projective plane $\Proj\; k[xyz, zw, v]$. Hence $\Pro^2$ is the good moduli space
of the quotient stack $[X/\G_m^2]$. In \cite[Theorem 1]{EGS:13} we showed that any choice of pushforward map for Artin stacks with good moduli spaces that commutes
with the inclusion of the smooth (and hence regularly embedded)
substacks $[V(x)/\G_m^2]$ and $[V(y)/\G_m^2]$ must be the 0 map.

Note however,
that neither $[V(x)/\G_m^2]$ nor $[V(y)/\G_m^2]$ is \emph{strongly} embedded. To see this, notice that $D(v)$ is an affine open subset of $X$ that is $\G_m^2$-invariant, and that $V(x)\cap D(v)$ is not defined by a $\G_m^2$-invariant function. Indeed the substack $[V(x)/\G_m^2]$ contains no stable or properly stable points, since every point with vanishing $x$-coordinate is in the saturation of the locus $V(x,y)$ which has positive dimensional stabilizer at all points. Similarly, $[V(y)/\G_m^2]$ is not strongly embedded.

By contrast, $[V(v)/G]$ is a strong substack. To see this note that $X$ is covered by the 3 affine open $\G_m^2$-invariant subsets $D(v)$, $D(xyz)$, and $D(zw)$. On each of these affine patches, $V(v)$ is defined by an invariant function: $V(v)\cap D(v)=\emptyset=V(1)$, $V(v)\cap D(xyz) = V(v/xyz)$, and $V(v)\cap D(zw) = V(v/zw)$.

\subsection{Sketch of the proof of Theorem \ref{thm.pushforward}}
The proof Theorem \ref{thm.pushforward} uses the following result of \cite{EdRy:16} which is a generalization of an earlier result for toric stacks in \cite{EdMo:12}.

\begin{theorem} \label{thm.main}
Let $\ix$ be a properly stable Artin stack with good moduli space $\bX$.
There is a canonical sequence of birational morphisms of smooth Artin stacks
$\ix_n \to \ix_{n-1} \to \ldots \to \ix_{0} = \ix$ with the following properties.
\begin{enumerate}
\item The stack $\ix_{n}$ is Deligne-Mumford. 

\item Each $\ix_k$ admits a good moduli space morphism $\ix_k \to \bX_k$ with $\bX_k$ an algebraic space. Moreover, $\ix_k^{ps} \neq \emptyset$.

\item The morphism $\ix_{k+1} \to \ix_{k}$ is an isomorphism over the maximal saturated DM substack $\ix_{k}^{ps}\subset\ix$ and it is an open immersion over the complement of a proper closed substack of $\ic_k \subset \ix_{k}$.

\item The morphism $\ix_{k+1} \to \ix_{k}$ induces a projective birational morphism of good moduli spaces $\bX_{k+1} \to \bX_{k}$.

\item The maximum dimension of the stabilizers of the points of $\ix_{k+1}$ is strictly smaller than that of $\ix_{k}$.
\end{enumerate}
\end{theorem}

\begin{remark}
  The birational morphisms $\ix_{k+1} \to \ix_{k}$ are called {\em
    Reichstein transformations}. They are defined as follows. Let
  $\ic_k$ be the substack of $\ix_k$ parametrizing points with
    maximal dimensional stabilizer. This locus is necessarily a closed
    smooth substack of $\ix$. Let $\is_k$ be the saturation of $\ic_k$
    with respect to the good moduli space morphism $\ix_{k} \to
    \bX_k$. Then $\ix_{k+1}$ is defined as the complement of the
    strict transform of $\is_k$ in the blow up of $\ix_{k}$ along
    $\ic_k$
\end{remark}
\begin{proof}[Proof of Theorem \ref{thm.pushforward}]
  The map $\pi_* \colon A^*(\ix)_\Q \to A^*(\bX)_\Q$ is defined as
  follows.  The morphisms $\ix_{k+1} \to \ix_{k}$ are representable
  morphisms of smooth stacks. In particular they are lci so, the
  composite morphism $f \colon \ix_n \to \ix$ is as well.  Hence there
  is a pullback of Chow group $f^* \colon A^*(\ix_{n}) \to
  A^*(\ix)$. On the other hand the morphisms of good moduli space
  $\bX_{k+1} \to \bX_{k}$ are birational and projective, so the composite map
$\fbar \colon \bX_n \to \bX$ is also birational and projective. Thus, 
there is a pushforward
  $\fbar_* \colon A^*(\bX_n) \to A^*(\bX)$. Since $\ix_n$ is a DM
  stack, we also have a pushforward of Chow groups $\pi_{n*} \colon
  A^*(\ix_n)_\Q \to A^*(\bX)_\Q$. We then define $\pi_* \colon
  A^*(\ix)_\Q \to A^*(\bX)_\Q$ as the composite $\fbar_* \circ \pi_{n*}
  \circ f^*$.

  Since the maps $\ix_{k+1} \to \ix_k$ are isomorphisms over the
  properly stable locus in $\ix$, if $\iz \subset \ix$ is a closed substack
  contained in the properly stable locus then $\pi_* [\iz]$ can be
  identified with $\pi_{n*} [\iz] = e_{\iz}^{-1} [\pi_n(\iz)] =
  e_{\iz}^{-1}[\pi(\iz)]$. This proves statement 2 of the theorem.
  Finally the last statement follows because the construction of
  \cite{EdRy:16} is functorial for strong lci morphisms.
\end{proof}

\begin{remark} A similar construction for a class of toric stacks was given in \cite{EdMo:13}.
\end{remark}
\section{Toward a theory of strong Chow groups: a conjectural answer to Question \ref{ques.subring}}
\label{sec:strong-Chow}

As discussed in Section \ref{sec.egs}, the example of \cite{EGS:13}
shows that there \emph{does not exist} a pushforward map $\pi_*\colon
A^*(\ix)\to A^*(X)$ that is functorial for {\em all} regular embeddings (let
alone all lci morphisms). So, rather than focusing on defining a
pushforward from $A^*(\ix)$, we focus on the subgroup generated by
strong cycles.

\begin{definition}
\label{def:strong-Chow}
Let $\pi:\ix\to \bX$ be a good moduli space map from an irreducible properly stable Artin stack. Then the {\em strong relative Chow group} $A^k_{st}(\ix/\bX)$ is the subgroup generated by $\sum c_i[\iz_i]$ with $\iz_i$ strong.
\end{definition}

We now state a series of conjectures concerning $A^k_{st}(\ix/\bX)$. Our first gives  a conjectural answer to Question \ref{ques.pushforward}.

\begin{conjecture} \label{conj.strongChow-push-forward}
If $\ix$ is a properly stable smooth Artin stack and $\pi\colon\ix\to\bX$ its good moduli space, then the assignment 
%Changed notation so as to not confuse this map with our previously defined pushforward. 
$[\iz] \mapsto e_\iz^{-1}[\pi(\iz)]$ for strong cycles $\iz$ respects rational equivalence, so we obtain a pushforward map $\pi_{st,*}\colon A^k_{st}(\ix/\bX)_\Q\to A^k(\bX)_\Q$.
\end{conjecture}

%Since the pushforward map $\pi_*$ of Theorem \ref{thm.pushforward} is defined on Chow groups, Conjecture \ref{conj.strongChow-push-forward} would be a consequence of the following conjecture.
We also state a stronger form of Conjecture \ref{conj.strongChow-push-forward}:

\begin{conjecture} \label{conj.Reichstein-push-forward} If $\ix$ is a
  properly stable smooth Artin stack and $\pi\colon\ix\to\bX$ its good
  moduli space, then the pushforward map $\pi_*\colon A^*(\ix)_\Q\to A^*(\bX)_\Q$ defined in Theorem
  \ref{thm.pushforward} satisfies $\pi_*[\iz]:=e_\iz^{-1}[\pi(\iz)]$ for strong
  cycles $\iz$.
\end{conjecture}

\begin{remark}
\label{rmk:conj-Reichstein-0-cycles}
Note that by definition, any strong 0-cycle must be contained in the stable locus so Theorem \ref{thm.pushforward} (\ref{item:stable-locus}) implies that Conjecture \ref{conj.Reichstein-push-forward} holds for strong 0-cycles.
\end{remark}

\begin{remark}
\label{rmk:conj-for-smooth-strong-substacks}
Conjecture \ref{conj.Reichstein-push-forward} is also true for smooth strong substacks $\iz\subseteq\ix$. Indeed, since $\iz$ is smooth, the inclusion map to $\ix$ is a strong regular embedding and so by Theorem \ref{thm.pushforward} (\ref{item:functoriality}), we reduce to the case that $\iz=\ix$ is the fundamental class. Since all maps $\ix_{i+1}\to\ix_i$ in Theorem \ref{thm.main} are birational, we therefore have $\pi_*[\iz] = e_\iz^{-1}[\iz]$.
\end{remark}

We now turn to Question \ref{ques.subring} where our main conjectures are as follows.

\begin{conjecture} \label{conj.strongChow-ring-structure} If $\ix$ is
  a properly stable smooth Artin stack with good moduli space $\bX$,
  then $\oplus_{k=0}^{\dim \ix} A^k_{st}(\ix/\bX)$ is a subring of
  $A^*(\ix)$ under the intersection product.
\end{conjecture}

\begin{conjecture} \label{conj.Chowsubring} Assuming Conjectures
  \ref{conj.strongChow-push-forward} and
  \ref{conj.strongChow-ring-structure}, then there is a subring
  $A^*_{inj}(\ix/\bX)_\Q$ of $A^*_{st}(\ix/\bX)_\Q$
which contains the subalgebra generated by
  $A^1_{st}(\ix/\bX)_\Q$ and has the property that the pushforward
  $\pi_{st,*}$ is injective on $A^*_{inj}(\ix/\bX)_\Q$. Moreover, if
  $\bX$ has only quotient singularities then $\pi_{st,*}$ is bijective
  on all of $A^*_{st}(\ix/\bX)_\Q$.
\end{conjecture}

\begin{remark}[{Conjectural answer to Question \ref{ques.subring}}]
\label{rmk:answering-ques.subring}
Notice that if Conjecture \ref{conj.Chowsubring} holds
then it provides an answer to Question \ref{ques.subring}: the image of $A^*_{inj}(\ix/\bX)$ under $\pi_{st,*}$ yields a non-trivial subgroup of $A^*(\bX)$ with an intersection product.
\end{remark}

\begin{remark} \label{rmk.strongChow-is-better} We will see in Example
  \ref{ex.quadriccone} that the assignment $[\iz] \mapsto
  e_{\iz}^{-1}[\iz]$ is not injective on {\em all} strong cycles. By
  analogy with the DM case, one might hope that it is possible to
  associate to every scheme $\bX$ with reductive quotient singularities
  (i.e.~those \'etale locally of the form $V/G$ where $V$ is a representation of
  a linearly reductive algebraic group) a {\em canonical} properly
  stable Artin stack $\ix$ whose good moduli space is $\bX$.  If this were
  the case, then assuming Conjecture  \ref{conj.Chowsubring},
 for every such scheme $\bX$, its Chow groups $A^*(\bX)$ are equipped with a canonical subring, namely $A^*_{inj}(\ix/\bX)$.
Moreover, since the stack $\ix$ is canonical the full
strong Chow ring $A^*_{st}(\ix/\bX)$ is also an invariant of the scheme $\bX$.
\end{remark}

\subsection{Examples illustrating the conjectures}

\begin{exam} %EdMo example

Let $G = \G_m^2$ act on $\A^4$ with weights $\left(\begin{array}{cccc}
1 & 0 & 1 & 1\\
0 & 1 & 0 & 1\end{array}\right)$. We let the coordinate functions on $\A^4$ be $(x_1,x_2,x_3,z)$.  Let $X = \A^4 \smallsetminus
V(x_1x_2, x_2x_3, z)$ and $\ix = [X/G]$. The map 
$X \to \Pro^2$ given by $(x_1, x_2, x_3, z) \mapsto (x_1x_2: x_2x_3 : z)$
is a good quotient.\footnote{This follows because $\A^4 \smallsetminus
V(x_1x_2, x_2x_3, x_4)$ is the $(1,1)$-semi-stable locus for the action of $G$
on $\A^4$ where $(1,1)$ is the character $(s,t) \mapsto st$. }
Hence $\ix\to \Pro^2$ is a good moduli space morphism. The maximal
saturated DM substack is the quotient $[X^s/G]$ and $X^s = X \smallsetminus V(x_1x_2, x_2x_3)$ since $V(x_1x_2, x_2x_3)$ is the saturation of the locus in $V(x_1,x_2, x_3)$ where $G$ acts with positive dimensional stabilizer. We verify all four of the above conjectures for this example.

If we denote by $s,t$ the first Chern classes of the projection characters of $\G_m^2$ then $$A^*(\ix) = A^*_{\G_m^2}(X) =\Z[s,t]/\left(t(s+t), s^2(s+t)\right).$$ We next compute the strong Chow groups of $\ix$. First note that since $\ix$ is non-singular, any Weil divisor is Cartier. Now if $[D]$ is the support of a strong Cartier divisor then $D = V(f)$ where $f$ is a function which is invariant on each $G$-invariant affine open in $X$.  Such a function must necessarily be a homogeneous polynomial in the semi-invariants $(x_1x_2, x_2x_3,z)$ and the Chow class of such a polynomial is a multiple of $s+t$. Moreover, the divisor $V(z)$, which has Chow class $s+t$, is strong. Hence $A^1_{st}(\ix/\bX) \simeq \Z$ generated by $s+t$.

Next, $A^2_{st}(\ix/\bX)$ is generated by $[V(x_1,z)]$
which is the class of a non-stacky closed point in $\ix$; its Chow class is $s
(s+t)$. Since $t(s+t) =0$ in $A^*(\ix)$ we see that $[V(x_1, z)] = [V(z)]^2$.
The relation $(s+t)^3 = 0$ implies $A^*_{st}(\ix/\bX)$ is closed under multiplication, verifying Conjecture \ref{conj.strongChow-ring-structure}. In fact, we have shown $$A^*_{st}(\ix/\bX)=\Z[s+t]/(s+t)^3$$ as rings.% which of course equals $A^*(\Pro^2).$

Since $V(z)$ is contained in the stable locus, $\pi_*[V(z)] =[\pi(V(z))]$ by Theorem \ref{thm.pushforward} (\ref{item:stable-locus}).  Now any irreducible strong divisor is of the form $V(f)$ where $f$ is an irreducible homogeneous polynomial of degree $d$ in the
semi-invariants $(x_1x_2, x_2x_3, z)$. Thus $[V(f)] = d [V(z)]$ so
$\pi_*[V(f)] = d h$, where $h$ is the hyperplane class on
$\Pro^2$. On the other hand, $\pi(V(f)) = V(f(A,B,C))$  where
$A,B,C$ are the projective coordinates on the quotient $\Pro^2$.
Thus $[\pi(V(f))] =
d h$ as well. Hence Conjecture \ref{conj.Reichstein-push-forward}
holds for strong divisors.  Moreover, any strong 0-cycle (i.e.~an
element of $A^2_{st}(\ix/\bX)$) is contained in the stable locus so
Conjecture \ref{conj.Reichstein-push-forward} holds for all strong cycles, and Conjecture \ref{conj.strongChow-push-forward} follows as well.

Finally, as shown above, $A^*_{st}(\ix/\bX)=\Z[s+t]/(s+t)^3$ and moreover, $\pi_*\colon A^*_{st}(\ix/\bX)\to A^*(\Pro^2)=\Z[h]/h^3$ sends $s+t$ to $h$. Thus, $\pi_*$ is an isomorphism on $A^*_{st}(\ix/\bX)$, verifying Conjecture \ref{conj.Chowsubring}.
\end{exam}
%We show that the strong Chow group is closed under intersection product, verifying Conjecture \ref{conj.strongChow-ring-structure} in this case. We also show that the pushforward $\pi_*$ of Theorem \ref{thm.pushforward} satisfies $\pi_*([\iz]) = [\pi(\iz)]$ for strong cycles and that $\pi_*$ is bijective on the strong Chow group. This implies Conjectures \ref{conj.strongChow-push-forward} and \ref{conj.Reichstein-push-forward} for this example as well.

\begin{remark}[{Generically strong cycles}]
\label{rmk:gen-strong}
Note that $[V(x_1)] + [V(x_2)] = s+t$. Since the image of $V(x_2)$ is
the point $(0:0:1)$ one might expect that $\pi_*[V(x_2)] = 0$ and thus $\pi_*[V(x_1)] = h$ is the hyperplane class. 

In this example, the cycle $V(x_1)$ is not strong, but the ``extra'' component in its saturation,
$V(x_2)$, does not dominate $\pi(V(x_1))$. We call a cycle $\iz \subset \ix$ {\em generically strong} if any extra components of the saturation of $\iz$ do not dominate $\pi(\iz)$.  We conjecture that $\pi_*([\iz]) = \pi(\iz)$ for all such cycles, which would strengthen Theorem \ref{thm.pushforward} (\ref{item:stable-locus}).
%In our previous example we expected that $\pi_*[V(x_1)] = h = [\pi(V(x_1))]$ which is consistent with this conjecture.
\end{remark}

\begin{exam}[{Good quotient with worse than quotient singularities}] \label{ex.quadriccone}
Consider the action of $\G_m^2$ on $\A^5$ with weight matrix
$$\left(\begin{array}{ccccc} 1 & 0 & 1 & 0 & 1\\0 & 1 & 0 & 1 & 1\end{array}
\right)$$ We denote the coordinates as $(x_1, x_2, x_3, x_4, z)$ and
let $X = \A^5 \smallsetminus V(x_1x_2, x_1x_4,x_2 x_3, x_3x_4,v)$.
The good quotient $X/G$ is $\Proj \;k[x_1x_2, x_1x_4, x_2x_3,
x_3x_4,v]$. This is the projective closure in $\Pro^4$ of the cone over the quadric hypersurface in $\Pro^3$; its singularity is not a quotient singularity. We again verify all four of the conjectures for this example modulo an assumption about the structure of $A^2_{st}(\ix/\bX)$.

% It can also be realized as the
% toric variety determined by the complete fan in $\R^3$ spanned by the
% rays $v_0=-e_1-e_2-e_3, v_1=e_1, v_2= e_2, v_3=e_3, v_4 = e_1 -e_2 +
% e_3$. The maximal cone spanned by $\{v_1,v_2, v_3, v_4\}$ is
% non-simplicial and the 4 maximal cones spanned by 
% $$\{v_0, v_1,v_2\},
% \{v_0,v_1,v_4\}, \{v_0, v_2, v_3\}, \{v_0, v_3, v_4\}$$
% respectively are non-singular.

% This stack has a single point with one dimensional stabilizer which
% correspondeds to the single orbit of $(0,0,0,0,z)$ in $X$. The saturation
% of this locus is the union of two linear subspaces $(x_1,x_3)$ and $(x_2, x_4)$,
% so $X^s = X \smallsetminus \left( V(x_1, x_3) \cup V(x_2, x_4)\right)$. 
% Blowing up this locus and deleting the saturations $V(x_1, x_3)$
% and $V(x_2, x_4)$ produces a toric DM stack whose coarse space is
% the toric variety obtained adding the ray $v=v_1 + v_2 + v_3 + v_4= 2e_1 + 2e_2$
% and subdiving the non-simplicial cone along this ray. (Such subdivisions are
% called {\em stacky star subdivisions} in \cite{EdMo:12}.)

% We now investigate our Conjectures for this example.

If we denote by $s,t$ the first Chern classes of the projection
characters of $\G_m^2$ then $$A^*(\ix) = A^*_{\G_m^2}(X) =
\Z[s,t]/\left(s^2(s+t), t^2(s+t)\right).$$ We now (conjecturally)
compute the strong relative Chow groups of $\ix$. First note that since $\ix$
is non-singular, any Weil divisor is Cartier. Now if $[D]$ is the
support of a strong Cartier divisor then $D = V(f)$ where $f$ is a
function which is invariant on each $G$-invariant affine open in $X$.
Such a function must necessarily be a homogeneous polynomial in the
semi-invariants $(x_1x_2, x_1x_4, x_2x_3, x_3 x_4, v)$ and the Chow class of any such polynomial is a multiple of $s+t$. Moreover, the divisor $V(v)$, which has Chow class $s+t$, is strong. Hence $A^1_{st}(\ix/\bX) \simeq\Z$. 

At the other extreme, $A^3_{st}(\ix/\bX)$ is generated by $[V(x_1,x_2, v)]$
which is the class of a non-stacky closed point in $\ix$. Its Chow class is $st(s+t)$.

For $A^2_{st}(\ix/\bX)$ we only have a conjectural description. If we
assume that the strong Chow group is generated by classes of substacks
which are regularly embedded then $A^2_{st}(\ix/\bX)$ is generated by the
equivariant classes of $V(x_1,z)$ and $V(x_2, z)$. (It is easy to
check that these are both strong cycles.)  Since $[V(x_1,z)] = s(s+t)$
and $[V(x_2, z)] = t(s+t)$, we have $A^2_{st}(\ix/\bX) = \Z [V(x_1,z)]  + \Z [V(x_2,z)]$.

If we let $\alpha$, $\beta$, and $\gamma$ be classes of $s+t$, $t(s+t)$, and $s(s+t)$ respectively, then $A^*_{st}(\ix/\bX)$ is closed
under multiplication in $A^*(\ix)$ and is equal to the ring
$$\Z[\alpha, \beta, \gamma]/\left( \alpha^2 -(\beta + \gamma),
  \alpha \beta -\alpha\gamma, \beta\gamma, \beta^2 \right).$$ 
Thus, modulo our assumption that $A^2_{st}(\ix/\bX)$ is generated by regularly embedded substacks we see that Conjecture \ref{conj.strongChow-ring-structure} holds.

We next verify Conjecture \ref{conj.Reichstein-push-forward}, and hence \ref{conj.strongChow-push-forward}. We have already shown that every strong divisor is rationally equivalent to an integer multiple of $V(v)$. Since $V(v)$ misses the stable locus, $\pi_*[V(v)] =[\pi(V(v))]$, so Conjecture \ref{conj.Reichstein-push-forward} holds for all strong divisors. By Remark \ref{rmk:conj-Reichstein-0-cycles}, the conjecture also holds for strong 0-cycles, so we need only consider strong curves. This follows from the observation that if we identify $X/G$ as the projective variety $\Proj\; k[A,B,C,D,V]/(AD - BC)$, then $\pi_*[V(x_1, v)]$ is the class of $A=B=V=0$ and $\pi_*[V(x_2,v)]$ is the class of $A=C=V=0$.
% then $\pi(V(v))$ is the Cartier divisor $V=0$. As in the previous example we can also conclude that for any strong divisor $V(f)$, $\pi_*[V(f)]= [\pi(V(f))]$, so Conjectures \ref{conj.strongChow-push-forward} and \ref{conj.Reichstein-push-forward} hold for divisors.

Lastly, we turn to Conjecture \ref{conj.Chowsubring}. Interestingly, $\pi_*$ is \emph{not} injective on strong cycles since the pushforwards of $\beta=[V(x_1, v)]$ and $\gamma=[V(x_2,v)]$ are both equal to the class of a line through the vertex of the cone of $\bX$. This shows that in this example, we cannot take $A_{inj}^*(\ix/\bX)$ equal to $A^*_{st}(\ix/\bX)$, however, we can take it to be the subring generated by $\alpha$. Then $\pi_*$ maps this ring injectively to the subring of $A^*(\bX)$ generated by powers of the hyperplane class.
%Note that their sum is the class of the complete intersection $A=V=0$ on $X/G$. This class is the pullback of $h^2$ to $A^*(X/G)$ where $h$ is the hyperplane class on $\Pro^4$. 
%However, Conjecture \ref{conj.Chowsubring} still holds because $\pi_*$ is injective on the subring of $A^*_{st}(\ix/\bX)$ generated by $\alpha$. (Recall the $\beta + \gamma = \alpha^2$.) Its image is the subring generated by powers of the pullback of the hyperplane class from $\Pro^4$.
\end{exam}
\begin{remark}
  This example shows (\emph{cf}.~Remark \ref{rmk.strongChow-is-better}) that
  the strong relative Chow groups $A^*_{st}(\ix/\bX)$, are
  a more refined geometric invariant of the good moduli
  space $\bX$ since the rulings of the quadric embedded in $\bX =
  \Proj k[A,B,C,D,V]/(AD -BC)$ can be distinguished in
  $A^*_{st}(\ix/\bX)$, as $\beta$ and $\gamma$, but not in $A^*_{st}(\bX)$.
\end{remark}

\begin{exam}[{Verifying the conjectures for \cite{EGS:13} example}]
\label{ex:EGS-continued}
%\subsection{The example of \cite{EGS:13} continued}
%In this subsection, we return to the example of \cite{EGS:13} and compute its strong Chow groups as well as $\pi_*$ from Theorem \ref{thm.pushforward}.
%In general it is not straightforward to explicitly compute the pushforward $\pi_*$ of Theorem \ref{thm.pushforward} even when the strong Chow groups $A^*_{st}(\ix/\bX)$ are known. We illustrate with the example of \cite{EGS:13} where it is relatively easy to verify Conjecture \ref{conj.strongChow-ring-structure}.

We use the notation of Section \ref{sec.egs}. A look at the weight matrix shows that the
action has generically one-dimensional stabilizer along the linear
subspace $V(x,y)$ and two-dimensional stabilizer along the subspace
$V(x,y,z,w)$.  Thus the stable locus 
%(also in the sense of GIT) 
for the action of $\G_m^3$ is the complement of the saturation of $V(x,y)$
which is the union of the 3 coordinate planes $V(x), V(y), V(z)$; i.e.~the quotient stack $[X^{s}/\G_m^3]$ is the maximal saturated DM
substack of $\ix$. Since the action of $\G_m^3$ on $X^s$ is free, the
good moduli space morphism restricts to an isomorphism on this open
set. If $A,B, C$ are coordinates on $\Pro^2$ corresponding to the
semi-invariant functions $xyz, zw, v$ respectively then $[X^s/\G_m^3]$
can be identified with the open set $\A^2= \Pro^2 \smallsetminus V(A)$.

We have 
$$A^*(\ix)=\Z[s,t,u]/\left(u(s+t+u),s(s+t)(s+t+u), t(s+t)(s+t+u)\right)$$ 
where $s,t,u$ denote the first Chern
classes of the 3 projection characters $\G_m^3 \to \G_m$. 
With this notation the coordinate hyperplanes $x,y,z,w,v$ have Chow
classes $s,t,u, s+t, s+t+u$ respectively corresponding to the weight of
the action on each coordinate.

%Since the quotient is $\Pro^2$ its Chow ring is just $\Z[h]/h^3$ where $h$ is the hyperplane class.

We next calculate the strong Chow groups. Any strong divisor is given by $V(f)$ where $f$ is a homogeneous polynomial
in the semi-invariant coordinates $(xyz, zw, v)$. This implies that the class of such a divisor is a multiple of $s+t+u$. Since $[V(v)]=s+t+u$, we see $A^1_{st}(\ix/\bX) \simeq \Z$ generated by this class. Next, $A^2_{st}(\ix/\bX)$ is generated by the class of a non-stacky point  $[V(w,v)] = (s+t)(s+t+u)$. Since $u(s+t+u) = 0$ we see that $[V(w,v)] = [V(v)]^2$ in $A^2(\ix)$. It is easy to show $(s+t+u)^3 = 0$ so $[V(v)]^3 = 0$. It follows that $A^*_{st}(\ix/\bX)$ is closed under multiplication and equals the ring $\Z[s+t+u]/(s+t+u)^3$, verifying Conjecture \ref{conj.strongChow-ring-structure}.

% which is of course the Chow ring of the good moduli space $\Pro^2$. Hence Conjecture \ref{conj.strongChow-ring-structure} again holds for this example.

Furthermore, since $V(v)$ is smooth and strong, Remark \ref{rmk:conj-for-smooth-strong-substacks} shows $\pi_*[V(v)] = [\pi(V(v))] = h$, the class of a hyperplane on $\Pro^2$. Combined with Remark \ref{rmk:conj-Reichstein-0-cycles}, this shows Conjecture \ref{conj.Reichstein-push-forward} holds for all strong cycles, and hence Conjecture \ref{conj.strongChow-push-forward} holds as well. Finally, since $\pi_*(s+t+u)=h$, notice that $\pi_*\colon A^*_{st}(\ix/\bX)\to A^*(\Pro^2)=\Z[h]/h^3$ is an isomorphism, verifying Conjecture \ref{conj.Chowsubring}.
\end{exam}

\begin{remark}
In Example \ref{ex:EGS-continued}, we were able to verify all four of the conjectures without ever calculating $\pi_*(s)$ and $\pi_*(t)$. We show how one may calculate these quantities, modulo the assumption that $\pi_*[V(z)]= 0$, N.B.~this seems like a reasonable assumption since $V(z)$ is in the saturation of the locus $(0:0:1)$, but it \emph{does not} follow from any of our conjectures. Note that $[V(v)]=[V(x)] + [V(y)] + [V(z)]$ and the automorphism $\ix \to \ix$ which exchanges $x$ and $y$ is a strong regular embedding, so $\pi_*[V(x)] = \pi_*[V(y)]$. On the other hand, $\pi(V(x)) = \pi(V(y)) = V(A)$. Since $\pi_*([V(x)] + [V(y)]) = [V(A)]$ we have $\pi_*[V(x)]= \pi_*[V(y)] = 1/2 [\pi(V(x))] = h/2$. Note that $V(x), V(y), V(z)$ are all contained in the complement of the stable locus.

In codimension 2 we have a similar calculation. The locus $V(w,v)$ consists of a single closed orbit whose image is the point $(1:0:0)$ in $\Pro^2$, so we know that $\pi_*([V(w,v)]) = [\pi(V(w,v))] = h^2$ which is the class of a point in $\Pro^2$. Since $[V(w,v)] = [V(x,v)] + [V(y,v)]$ we know that $\pi_*[V(x,w)] =\pi_*[V(y,w)] = h^2/2$. 
\end{remark}

%\begin{remark}\matt{notice that we have verified all of the conjectures in all of the examples, modulo an assumption in second example concerning $A^2_{st}$. Furthermore, we point out that in the examples where $\bX$ has quotient singularities, $\pi_{st,*}$ is injectve, illustratuing the last part of Conjecture...}\end{remark}

% We need to describe the process of Theorem \ref{thm.main} that prodces a
% Deligne-Mumford stack.  

% To obtain the canonical DM stack associated to $\ix$ we perform two
% Reichstein transforms. For the first we blow up $X$ along $V(x,y,z,w)$
% and then delete the strict transforms of its saturation which is 
% $V(z), V(x,w), V(y,w)$. The new quotient stack 
% new stack has 1-dimensional stabilizer along the exceptional divsiors
% of the strict transforms of the linear subspaces $(x,v)$ and $(y,v)$
% as well as long the strict transform of the linear subspace $V(x,y)$.
% \footnote{Note that the Reichstein transform contains points in the exceptional
% divisor $\Pro^3 \times \G_m$ with $\Z_2$ stabilizers such as
% $\left((1:1:1:0), 0\right)$.}

% The second Reichstein transform blows up these loci and then deletes the strict transforms of their corresponding saturations.

%\bibliographystyle{../amsmath}
%\bibliography{refs}
%\def\cprime{$'$}
%\def\cprime{$'$}
\def\cprime{$'$} \def\cprime{$'$} \def\cprime{$'$}

\myaddress{Department of Mathematics, University of Missouri, Columbia MO 65211}
\myemail{edidind@missouri.edu}

\myaddress{Pure Mathematics, University of Waterloo, 200 University Avenue West, Waterloo, Ontario, Canada N2L 3G1}
\myemail{msatriano@uwaterloo.ca}

\end{document}